\documentclass[11 pt, reqno]{amsart}
\usepackage{color}
\usepackage{verbatim}
\usepackage{amssymb}
\usepackage{amscd}
\usepackage{amsmath}
\usepackage{amsthm}
\usepackage{eucal}
\usepackage{setspace}
\usepackage{url}
\usepackage[utf8]{inputenc}
\usepackage{booktabs}

\usepackage[hyperfootnotes=false, colorlinks, linkcolor={blue}, citecolor={magenta}, filecolor={blue}, urlcolor={blue}]{hyperref}

\allowdisplaybreaks

\renewcommand{\H}{\mathbb H}
\newcommand{\leftexp}[2]{{\vphantom{#2}}^{#1}%
      \kern-\scriptspace%
      {#2}}

\renewcommand{\b}{{\mathbf b}}

\renewcommand{\L}{{\mathcal L}}

\newcommand{\Tr}{\mathrm{Tr}}
\newcommand{\Q}{{\mathbb Q}}
\newcommand{\Z}{{\mathbb Z}}
\newcommand{\F}{{\mathbb F}}
\newcommand{\R}{{\mathbb R}}
\newcommand{\C}{{\mathbb C}}
\newcommand{\bs}{\backslash}

\newcommand{\GL}{{\rm GL}}

\newcommand{\SL}{{\rm SL}}
\newcommand{\SO}{{\rm SO}}

\newcommand{\Sp}{{\rm Sp}}

\newcommand{\sym}{{\rm sym}}

\newcommand{\HH}{{\mathbb H}}
\newcommand{\mat}[4]{{\setlength{\arraycolsep}{0.5mm}\left[\begin{array}{cc}#1&#2\\#3&#4\end{array}\right]}}

\newtheorem{lemma}{Lemma}[section]
\newtheorem{theorem}[lemma]{Theorem}

\newtheorem{corollary}[lemma]{Corollary}
\newtheorem{proposition}[lemma]{Proposition}
\newtheorem{definition}[lemma]{Definition}
\theoremstyle{remark}

\begin{document}

\bibliographystyle{plain}

\title[On the growth of nearly holomorphic forms]{A note on the growth of nearly holomorphic vector-valued Siegel modular forms}

\author{Ameya Pitale}
\address{Department of Mathematics
\\ University of Oklahoma\\ Norman\\
   OK 73019, USA}
\email{apitale@math.ou.edu}

\author{Abhishek Saha}
\address{Departments of Mathematics \\
  University of Bristol\\
  Bristol BS81SN \\
  UK} \email{abhishek.saha@bris.ac.uk}

\author{Ralf Schmidt}
\address{Department of Mathematics
\\ University of Oklahoma\\ Norman\\
   OK 73019, USA}
\email{rschmidt@math.ou.edu}

\begin{abstract}
Let $F$ be a nearly holomorphic vector-valued Siegel modular form of weight $\rho$ with respect to some congruence subgroup of $\Sp_{2n}(\Q)$. In this note, we prove that the function on $\Sp_{2n}(\R)$ obtained by lifting $F$ has the moderate growth (or ``slowly increasing") property. This is a consequence of the following bound that we prove: $\|\rho(Y^{1/2})F(Z) \| \ll \prod_{i=1}^n (\mu_i(Y)^{\lambda_1/2} + \mu_i(Y)^{-\lambda_1/2})$ where $ \lambda_1 \ge \ldots \ge \lambda_n$ is the highest weight of $\rho$ and $\mu_i(Y)$ are the eigenvalues of the matrix $Y$.
\end{abstract}

 \maketitle

 \section{Introduction and statement of result}
Let $G$ be a connected reductive group over $\Q$ and $K$ a maximal compact subgroup of $G(\R)$.   One of the properties that an automorphic form on $G(\R)$ is required to satisfy is that it should be a slowly increasing function (also referred to as the \emph{moderate growth property}). We now recall the definition of this property following \cite{BorelJacquet1979}.

A norm $\| \ \|$ on $G(\R)$ is a function of the form $\|g \| = (\Tr( \sigma(g)^\ast \sigma(g)))^{1/2}$ where $\sigma: G(\R) \rightarrow \GL_r(\C)$ is a finite-dimensional representation with finite kernel and image closed in $M_r(\C)$ and such that $\sigma(K) \subseteq \SO_m$. For example, if $G=\Sp_{2n}$, we may take $\sigma$ to be the usual embedding into $\GL_{2n}(\R)$ while for $G = \GL_n$ we may take $\sigma(g) = (g, \det(g)^{-1})$ into $\GL_{n+1}(\R)$. A complex-valued function $\phi$ on $G(\R)$ is said to have the moderate growth property if there is a norm $\| \ \|$ on $G(\R)$, a constant $C$, and a positive integer $\lambda$ such that $|\phi(g)| \le C \|g\|^\lambda$ for all $g \in G(\R)$. This definition does not depend on the choice of norm.

In practice, automorphic forms on $G(\R)$ are often constructed from classical objects (such as various kinds of ``modular forms") and it is not always immediately clear that the resulting constructions satisfy the moderate growth property. For a classical modular form $f$ of weight $k$ on the upper half plane, one can prove the bound $|f(x + iy)| \le C (1 + y^{-k})$
for some constant $C$ depending on $f$. Using this bound it is easy to show that the function $\phi_f$ on $\SL_2(\R)$ attached to $f$ has the moderate growth property. More generally, if $F$ is a holomorphic Siegel modular form of weight $k$ on the Siegel upper half space $\H_n$, Sturm proved the bound $|F(X + iY)| \le C \prod_{i=1}^n (1 + \mu_i(Y)^{-k})$ where $\mu_i(Y)$  are the eigenvalues of $Y$, which can be shown to imply the moderate growth property for the corresponding function $\Phi_F$ on $\Sp_{2n}(\R)$.

Bounds of the above sort are harder to find in the literature for more general modular forms. In particular, when considering  Siegel modular forms on $\H_n$, it is more natural to consider vector-valued modular forms. Such a vector-valued form comes with a representation $\rho$ of $\GL_n(\C)$ corresponding to a highest weight $\lambda_1 \ge \ldots \ge \lambda_n \ge 0$ of integers. Furthermore, for arithmetic purposes, it is sometimes important to consider more general modular forms where the holomorphy condition is relaxed to near-holomorphy. Recall that a nearly holomorphic modular form on $\H_n$ is a function that transforms like a modular form, but instead of being holomorphic, it is a polynomial in the the entries of $Y^{-1}$ with holomorphic functions as coefficients. The theory of nearly holomorphic modular forms was developed by Shimura in substantial detail and was exploited by him and other authors to prove algebraicity and Galois-equivariance of critical values of various $L$-functions. We refer the reader to the papers~\cite{bluher, bochpilot96, heimboch, sah2, ShimuraHilbert, shimura2000} for some examples.

We remark that the moderate growth property for a certain type of modular form is absolutely crucial if one wants to use general results from the theory of automorphic forms to study these objects (as we did in our recent paper \cite{PSS14} in a certain case). It appears that a proof of the moderate growth property, while probably known to experts, has not been formally written down in the setting of nearly holomorphic vector-valued forms.  In this short note, we fill this gap in the literature.

Consider a nearly holomorphic vector-valued modular form  $F$ of highest weight $\lambda_1 \ge \ldots \ge \lambda_n \ge 0$ with respect to a congruence subgroup. The function $F$ takes values in a finite dimensional vector space $V$. For $v \in V$, denote $\| v \| = \langle v, v \rangle^{1/2}$ where we fix an $U(n)$-invariant inner product on $V$. We can lift $F$ to a $V$-valued function $\vec \Phi_F$ on $\Sp_{2n}(\R)$. For any linear functional $\L$ on $V$ consider the complex valued function $\Phi_F = \L \circ \vec \Phi_F$ on $\Sp_{2n}(\R)$. We prove the following result.
 \begin{theorem}\label{t:main}The function $\Phi_F$ defined above has the moderate growth property.
 \end{theorem}

The above theorem is a direct consequence of the following bound.

\begin{theorem}\label{t:boundfintro}For any nearly holomorphic vector-valued modular form  $F$ as above, there is a constant $C$ (depending only on $F$) such that for all $Z = X+iY \in \H_n$ we have $$\|\rho(Y^{1/2})F(Z) \| \le C \prod_{i=1}^n (\mu_i(Y)^{\lambda_1/2} + \mu_i(Y)^{-\lambda_1/2}).$$
\end{theorem}
 The proof of Theorem \ref{t:boundfintro}, as we will see, is extremely elementary. It uses nothing other than the existence of a Fourier expansion, and is essentially a straightforward extension of arguments that have appeared in the classical case, e.g., in \cite{maassbook} or \cite{sturm}. This argument is very flexible and can be modified to provide a bound for Siegel-Maass forms. With some additional work  (which we do not do here), Theorem \ref{t:boundfintro} can be used to derive a bound on the Fourier coefficients of $F$. We also remark that the bound in Theorem \ref{t:boundfintro} can be substantially improved if $F$ is a cusp form.

\subsection*{Notations}
For a positive integer $n$ and a commutative ring $R$, let $M_n^{\sym}(R)$ be the set of symmetric $n\times n$ matrices with entries in $R$. For $X,Y\in M_n^{\sym}(\R)$, we write $X>Y$ if $X-Y$ is positive definite. Let $\H_n$ be the Siegel upper half space of degree $n$, i.e., the set of $Z=X+iY\in M_n^{\sym}(\C)$ whose imaginary part $Y$ is positive definite. For such $Z$ and $g=\mat{A}{B}{C}{D}\in\Sp_{2n}(\R)$, let $J(g,Z)=CZ+D$.

For any complex matrix $X$ we denote by $X^*$ its transpose conjugate. For positive definite $Y=(y_{ij})\in M_n^{\sym}(\R)$, let $\|Y\| = \max_{i,j} |y_{ij}|$. We denote by $\mu_i(Y)$ the $i$-th eigenvalue of $Y$, in decreasing order.
\section{Nearly holomorphic functions and Fourier expansions}
For a non-negative integer $p$, we let $N^p(\HH_n)$ denote the space of all polynomials in the entries of $Y^{-1}$ with total degree $\le p$  and with holomorphic functions on $\HH_n$ as coefficients. The space
$$
 N(\HH_n)=\bigcup_{p\geq0}N^p(\HH_n)
$$
is the space of \emph{nearly holomorphic functions} on $\H_n$.

It will be useful to have some notation for polynomials in matrix entries. Let  $$R_n = \{(i,j): 1 \le i \le j \le n\}.$$ Let $$T_n^p = \{\mathbf{b} = (b_{i,j}) \in \Z^{R_n}:b_{i,j} \ge 0, \sum_{(i,j)\in R_n} b_{i,j} \le p\}.$$  For any $V = (v_{i,j}) \in M_n^{\sym}(\R)$, and any $\b \in T_n^p$, we define $[V]^{\mathbf{\b}} = \prod_{(i,j) \in R_n }  v_{i, j}^{b_{i,j}}$. In particular, a function $F$ on $\H_n$  lies in $N^p(\H_n)$ if and only if there are holomorphic functions $G_\b$ on $\H_n$ such that $$F(Z) = \sum_{\b \in T_n^p} G_\b(Z) [Y^{-1}]^\b.$$

\begin{definition}For any $\delta>0$, we define
$$V_\delta = \{Y \in M_n^{\sym}(\R): Y \ge \delta I_n \}.$$
\end{definition}

\begin{lemma}\label{lem:bddvert} Given any $Y \in V_\delta$, we have $\|Y^{-1}\| \le \delta^{-1}$.
\end{lemma}
\begin{proof}Note that for any positive definite matrix $Y' = (y'_{ij})$ we have $\|Y' \| = \max_{i,j} |y'_{ij}| =  \max_{i} y'_{ii}$. This is an immediate consequence of the fact that each $2 \times 2$ principal minor has positive determinant and each diagonal entry is positive.

So it suffices to show that each diagonal entry of $Y^{-1}$ is less than or equal to $\delta^{-1}$.
But the assumption $Y \ge \delta I_n$ implies that $Y^{-1} \le \delta^{-1}I_n$, which implied the desired fact above.
\end{proof}
An immediate consequence of this lemma is that for any $\delta\le 1$, $Y \in V_\delta$ and  $\b \in T_n^p$, we have $|[Y^{-1}]^\b| \le \delta^{-p}$.

\begin{definition}We say that $F \in N^p(\H^n)$ has a nice Fourier expansion if there exists an integer $N$ and complex numbers $a_{\b}(F,S)$ for all $0 \le S \in \frac1N M_n^{\sym}(\Z)$, such that  we have  an expression
$$
 F(Z) = \sum_{\b \in T_n^p}\,\sum_{\substack{S \in \frac1N M_n^{\sym}(Z)\\S\ge 0}} a_{\b}(F,S) e^{2 \pi i \Tr(SZ)} [Y^{-1}]^\b
$$
that converges absolutely and uniformly on compact subsets of $\H_n$.
\end{definition}
Note that a key point in the above definition is that the sum is taken only over positive semidefinite matrices.
The next proposition, which is well-known in the holomorphic case, shows that this implies a certain boundedness property for the function $F$.

\begin{proposition}\label{p:koecher}Let $F \in N^p(\H^n)$ have a nice Fourier expansion. Then for any $\delta>0$, the function $F(Z)$ is bounded in the region
$\{Z = X +i Y: Y \in V_\delta\}$.
\end{proposition}
\begin{proof}We may assume that $\delta \le 1$.
By the notion of a nice Fourier expansion,  for each $\b \in T_n^p$, the series $$R_{\b}(Y):=\sum_{\substack{S \in \frac1N M_n^{\sym}(Z)\\S\ge 0}} |a_{\b}(F,S)| e^{-2 \pi \Tr(SY)}$$ converges for any $0<Y \in M_n^{\sym}(\R)$. For any $Z$ in the given region, using Lemma \ref{lem:bddvert}, we get
\begin{equation}\label{e:feabs}
|F(Z)| \le  \sum_{\b \in T_n^p} R_{\b}(Y) \delta^{-p},
\end{equation}
and so to prove the proposition it suffices to show that each $R_{\b}(Y)$ is bounded in the region $Y \ge \delta$. By positivity, we have $$|a_{\b}(F,S)\ e^{-2\pi \,{\rm Tr}(SY)}|\le R_{\b}(Y)$$ for each $Y>0$ and each $S\in \frac1N M_n^{\sym}(\Z)$. Therefore \begin{equation}\label{e:febd}|a_{\b}(F,S)| \le R_{\b}(\delta I_n/2)e^{\delta \pi \,{\rm Tr}(S)}.\end{equation} Next, note that if $Y \ge \delta I_n$, then ${\rm Tr}(SY) \ge \delta {\rm Tr}(S)$ for all $S \ge 0$. To see this, we write $Y= \delta I_n + Y_1^2$ where $Y_1 \ge 0$ is the square-root of $Y-\delta I_n$. As $S \ge 0$ we have ${\rm Tr}(Y_1SY_1) \ge 0$ and consequently ${\rm Tr}(SY) = {\rm Tr}(S\delta I_n) + {\rm Tr}(Y_1SY_1) \ge {\rm Tr}(S\delta I_n)$.

 Using the above and \eqref{e:febd}, we have for all $Y \ge \delta I_n$ $$R_{\b}(Y) \le R_{\b}(\delta I_n/2) \sum_{0 \le S \in \frac1N M_n^{\sym}(\Z)}
e^{-\delta \pi \,{\rm Tr}(S)}.$$  As the sum $\sum_{0 \le S \in \frac1N M_n^{\sym}(\Z)}
e^{-\delta \pi \,{\rm Tr}(S)}$ converges to a finite value (for a proof of this fact, see \cite[p. 185]{maassbook}) this completes the proof that $R_{\b}(Y)$ is bounded in the region
$Y \ge \delta I_n$.

\end{proof}

 \section{Bounding nearly holomorphic vector-valued modular forms}
Let $(\rho, V)$ be a finite-dimensional rational representation of $\GL_n(\C)$ and $\langle, \rangle$ be a (unique up to multiples) $U(n)$-invariant inner product on $V$. In fact, the inner product $\langle, \rangle$ has the property that
$$
\langle \rho(M) v_1, v_2\rangle = \langle  v_1, \rho(M^\ast) v_2\rangle
$$
for all $M \in \GL_n(\C)$. (To see this, note that it's enough to check it on the Lie algebra level. It's true for the real subalgebra $u(n)$ and so by linearity it follows for all of $\rm{gl}(n,\C)$.) For any $v \in V$, we define $$\|v \| = \langle v, v \rangle^{1/2}.$$

As is well known, the representation $\rho$ has associated to it an $n$-tuple $\lambda_1 \ge \lambda_2 \ge \ldots \ge \lambda_n \ge 0$ of integers known as the highest weight of $\rho$. We let $d_\rho$ denote the dimension of $\rho$.

We define a right action of $\Sp_{2n}(\R)$ on the space of smooth $V$-valued functions on $\H_n$ by
\begin{equation}\label{slashoperatoreq}
 (F\big|_{\rho}g)(Z)=\rho(J(g,Z))^{-1}F(gZ)\qquad\text{for }g\in\Sp_{2n}(\R),\;Z\in\HH_n.
\end{equation}

A \emph{congruence subgroup} of $\Sp_{2n}(\Q)$  is a subgroup that is commensurable with $\Sp_{2n}(\Z)$ and contains a principal congruence subgroup of $\Sp_{2n}(\Z)$. For a congruence subgroup $\Gamma$ and a non-negative integer $p$, let $N^p_{\rho}(\Gamma)$ be the space of all functions $F:\HH_n\to V$ with the following properties.
\begin{enumerate}
 \item For any $g\in\Sp_{2n}(\Q)$ and any linear map $\L: V \rightarrow \C$, the function $\L \circ (F\big|_{\rho}g)$ lies in $N^p(\H_n)$ and has a nice Fourier expansion.
 \item $F$ satisfies the transformation property \begin{equation}\label{modformeq1}
 F\big|_{\rho}\gamma=F\qquad\text{for all }\gamma\in\Gamma.
\end{equation}
\end{enumerate}
Let $N_{\rho}(\Gamma)=\bigcup_{p\geq0}N^p_{\rho}(\Gamma)$. We refer to $N_{\rho}(\Gamma)$ as the space of \emph{nearly holomorphic Siegel modular forms} of weight $\rho$ with respect to $\Gamma$. We put $N^{(n)}_{\rho} = \bigcup_\Gamma N_{\rho}(\Gamma),$  the space of all nearly holomorphic Siegel modular forms of weight $\rho$.

Recall that for any $Y>0$ in $M_n^{\sym}(\R)$, we let $\mu_1(Y) \ge \mu_2(Y) \ge \ldots \ge \mu_n(Y)>0$ denote the eigenvalues of $Y$. We can now state our main result.
\begin{theorem}\label{t:boundf}For any $F \in N^{(n)}_{\rho}$, there is a constant $C_F$ (depending only on $F$) such that for all $Z = X+iY \in \H_n$ we have $$\|\rho(Y^{1/2})F(Z) \| \le C_F \prod_{i=1}^n (\mu_i(Y)^{\lambda_1/2} + \mu_i(Y)^{-\lambda_1/2}).$$
\end{theorem}

In order to prove this theorem, we will need a couple of lemmas.

\begin{lemma}\label{weightineq}
 For any $v \in V$, and any $Y>0$ in $M_n^{\sym}(\R)$, we have

 $$
  \left(\prod_{i=1}^n \mu_i(Y)^{\lambda_{n+1-i}}\right)\|v\| \le \|\rho(Y) v \| \le \left(\prod_{i=1}^n \mu_i(Y)^{\lambda_i}\right) \|v\|.
 $$
\end{lemma}
\begin{proof} This follows from considering a basis of weight vectors. Note that it is sufficient to prove the inequalities for $Y$ diagonal as any $Y$ can be diagonalized by a matrix in $U(n)$ and our norm is invariant by the action of $U(n)$.
\end{proof}

Next, we record a result due to Sturm.

\begin{lemma}[Prop. 2 of \cite{sturm}]\label{l:sturm}Suppose that $\F$ is a fundamental domain for $\Sp_{2n}(\Z)$ such that there is some $\delta>0$ such that  $Z = X +i Y\in \F$ implies that $Y \in V_\delta$. Let $\phi: \H_n \rightarrow \C$ be any function such that there exist constants $c_1>0$, $\lambda \ge 0$ with the property that $| \phi(\gamma Z)| \le c_1 \det(Y)^\lambda$ for all $Z \in \F$ and $\gamma \in \Sp_{2n}(\Z)$. Then for all $Z \in \H_n$ we have the inequality
 $$
  |\phi(Z)| \le c_\phi \prod_{i=1}^n (\mu_i(Y)^{\lambda} + \mu_i(Y)^{-\lambda}).
 $$
\end{lemma}
\begin{proof}[Proof of Theorem \ref{t:boundf}]
Let $F$ be as in the statement of the theorem, so that $F \in N_\rho(\Gamma)$ for some $\Gamma \subset \Sp_{2n}(\Z)$. We let $\gamma_1, \gamma_2, \ldots, \gamma_t$ be a set of representatives for $\Gamma \bs \Sp_{2n}(\Z)$. Fix an orthonormal basis $v_1, v_2, \ldots, v_d$ of $V$ and for any $G \in N_\rho^{(n)}$ define $G_i(Z) := \langle G(Z), v_i \rangle$. Note that $\|G(Z)\| =
\left(\sum_{i=1}^d |G_i(Z)|^2\right)^{1/2}$.

Let $\F$ be as in Lemma \ref{l:sturm}.  By Proposition \ref{p:koecher}, it follows that there is a constant $C$ depending on $F$ such that $|(F|_{\rho} \gamma_r)_i(Z)| \le C$ for all $1 \le r\le t$, $1\le i \le n$, and $Z \in \F$. Moreover,  for any $Z = X+iY \in \F$ we have each $\mu_j(Y^{1/2}) \ge \delta^{1/2}$ and therefore $\left(\prod_{j=1}^n \mu_j(Y^{1/2})^{\lambda_j}\right) \le \det(Y)^{\lambda_1/2} \delta^{\frac12\sum_{j=2}^n (\lambda_j - \lambda_1)}.$ Now consider the function $\phi(Z) = \|\rho(Y^{1/2})F(Z) \|$.  For any $\gamma \in \Sp_{2n}(\Z)$, there exists $\gamma_0 \in \Gamma$ and some $1 \le r \le t$ such that $\gamma = \gamma_0 \gamma_r$. An easy calculation shows that
$$
 \|\phi(\gamma Z)\| = \|\rho(Y^{1/2})(F|_{\rho} \gamma_r)(Z) \|.
$$
So for all $Z \in \F$, $\gamma \in \Sp_{2n}(\Z)$ we have, using Lemma \ref{weightineq} and the above arguments,
$$
 \|\phi(\gamma Z) \| \le \det(Y)^{\lambda_1/2}  \delta^{\frac12\sum_{i=2}^n (\lambda_i - \lambda_1)} d^{1/2} C.
$$
So the conditions of Lemma \ref{l:sturm} hold with $\lambda = \lambda_1/2$. This concludes the proof of Theorem \ref{t:boundf}.
\end{proof}

\begin{corollary}For any $F \in N^{(n)}_{\rho}$, there is a constant $C_F$ (depending only on $F$) such that for all $Z = X+iY \in \H_n$ we have $$\|\rho(Y^{1/2})F(Z) \| \le C_F \ (1 + \Tr(Y))^{n \lambda_1} \ (\det Y)^{- \lambda_1/2}.$$
\end{corollary}
\begin{proof}This follows immediately from Theorem \ref{t:boundf} and the following elementary inequality, which holds for all positive integers $\lambda, n$ and  all positive reals $y_1, \ldots , y_n$:
$$
\prod_{i=1}^n (1 + y_i^{\lambda}) \le (1+ y_1 +\ldots +y_n)^{n \lambda}.
$$
To prove the above inequality, note that $1 + y_i^{\lambda}\le (1+ y_1 +\ldots +y_n)^{\lambda}$ for each $i$. Now take the product over $1 \le i \le n$.
\end{proof}
\section{The moderate growth property}
Given any $F \in N_{\rho}^{(n)}$, we define a smooth function $\vec\Phi_F$ on $\Sp_{2n}(\R)$ by the formula  $$\vec\Phi_F(g)=\rho(J(g,I))^{-1}F(gI),$$  where $I :=iI_{2n}.$

\begin{proposition}\label{Phivecprop}
Let $F \in N_{\rho}^{(n)}$ and $\vec\Phi_F$ be defined as above. Then there is a constant $C$  such that for all $Z = X+iY \in \H_n$ we have $$\left \|\vec\Phi_F\left(\mat{Y^{1/2}}{X Y^{1/2}}{}{Y^{-1/2}} \right)\right \| \le C \prod_{i=1}^n (\mu_i(Y)^{\lambda_1/2} + \mu_i(Y)^{-\lambda_1/2}).$$
\end{proposition}
\begin{proof}This follows immediately from Theorem \ref{t:boundf}.
\end{proof}

A complex-valued function $\Phi$ on $\Sp_{2n}   (\R)$ is said to be slowly increasing if there is a constant $C$ and a positive integer $r$ such that $$|\Phi(g)| \le C (\Tr(g^* g))^{r}$$ for all $g \in \Sp_{2n}(\R)$.

\begin{theorem}\label{modgrowththeorem}
Let $F \in N_{\rho}^{(n)}$ and $\vec\Phi_F$ be as defined above. For some linear functional $\L$ on $V$, let $\Phi_F = \L \circ \vec\Phi_F$. Then the function $\Phi_F$ has the moderate growth property.
\end{theorem}
\begin{proof}Note that $|\Phi_F(g)| \le \| \L \| \ \|\vec\Phi_F(g) \|$. So it suffices to show that there is a constant $C$ and a positive integer $r$ such that
\begin{equation}\label{modgrowththeoremeq}
 \|\vec\Phi_F(g) \| \le C (\Tr(g^* g))^{r}
\end{equation}
for all $g \in \Sp_{2n}(\R)$. Since both sides of this inequality do not change when $g$ is replaced by $gk$, where $k$ is in the standard maximal compact subgroup of $\Sp_{2n}(\R)$, we may assume that $g$ is of the form $\mat{Y^{1/2}}{X Y^{1/2}}{}{Y^{-1/2}}$. Then the existence of appropriate $C$ and $r$ follows easily from Proposition \ref{Phivecprop}. Indeed, we can take any $r \ge n \lambda_1/2$ and $C$ the same constant as in Proposition \ref{Phivecprop}.
\end{proof}
\bibliography{growth}{}
\end{document}